\documentclass[11pt]{article}
\usepackage{geometry}                
\usepackage{graphicx}
\usepackage{amssymb}
\usepackage{epstopdf}
\usepackage{amsthm}
\usepackage{amsmath}
\usepackage{mathtools}
\usepackage{txfonts}
\usepackage{url}
\newtheorem{theorem}{Theorem}[section]
\newtheorem{lemma}[theorem]{Lemma}
\newtheorem{proposition}[theorem]{Proposition}
\newtheorem{corollary}[theorem]{Corollary}

\title{On the Difference between Consecutive Primes and Estimates of the Number of Primes in the Interval $(n, 2n)$}
\author{Felix Sidokhine}

\begin{document}
\maketitle

\begin{abstract}
\noindent Using evaluations of the difference between consecutive primes we develop another way of estimating of the number of primes in the interval $(n, 2n)$. We also discuss the ultra Cramer conjecture, $p_{n+1} - p_n = O(log^{1+\epsilon}p_n)$ where $\epsilon > 0$, in the context of the results we have obtained in our paper.
\end{abstract}

\section{Introduction}
The difference between consecutive primes is an important characteristic of the distribution of the prime numbers \cite{Granville:2010aa}. However, as it will be illustrated, it is also closely linked with the estimates of the number of primes in the interval $(n, 2n)$. Using the best available evaluation $p_{n+1} - p_n = O(p_n^{0.525})$ \cite{Baker:2001aa} and also the hypothetical evaluations of the difference between consecutive primes $p_{n+1}- p_n = O(\sqrt{p_n}), p_{n+1}- p_n = O(\ln^2(p_n))$ \cite{Cramer:1936aa} we develop another way evaluating the number of primes in the interval $(n, 2n)$.

\section{Estimates of the number of primes in the interval $(n, 2n)$}

There is the well-known estimate for the number of the prime numbers in an interval $(n, 2n)$\cite{Trost:1968aa}:
\begin{equation}
\frac{1}{3}\frac{n}{\log(2n)} < \pi(2n) - \pi(n) < \frac{7}{5}\frac{n}{\log(n)} \text{ where } n > 1
\end{equation}


\noindent The left side is a lower bound for the number of primes within $(n,2n)$, and the right side is the upper bound.

\begin{proposition}\label{prop_a_1}
Let $k$, $N_k$ be such that for every prime $p_m \geq N_k$, $\sqrt{p_m} - \sqrt{p_{m-1}} < \frac{1}{k}$ is true. Then the interval $(n, 2n)$ contains no less than $[(\frac{k}{4})\sqrt{2n}]$ primes for every $n > N_k$. 
\end{proposition}

\begin{lemma}\label{lemma_a_1}
Let an integer $k$ be such that the condition of proposition \ref{prop_a_1} is satisfied. Then the interval $(\frac{p_n}{2}, p_n)$, for every $p_n$ where $\frac{p_n}{2} > N_k$, contains no less than $[(\frac{k}{4})\sqrt{p_n}]$ prime numbers.
\end{lemma}

\begin{proof}
Let lemma \ref{lemma_a_1} be true. Let $p_n$ be such a prime, $\frac{p_n}{2} > N_k$, that $(\frac{p_n}{2},p_n)$ contains less than $[(\frac{k}{4})\sqrt{p_n}]$ primes. Let $\{p_i\}$ be the set of all primes inside $(\frac{p_n}{2},p_n)$. Let $(\frac{p_n}{2},p_n) = \cup_{j=1}^{|\{p_i\}|} I_j = (\frac{p_n}{2},p_1) \cup (\cup_{j=1}^{|\{p_i\}|} (p_j,p_{j+1}))$. 
\newline
\newline
Some $I_i$ has length no less than $2\frac{\sqrt{p_n}}{k} <\frac{(p_n- \frac{p_n}{2})}{[(\frac{k}{4})\sqrt{p_n}]}< p_i - p_{i-1}$. This contradicts the condition \ref{prop_a_1} since $p_i - p_{i-1}< \frac{2\sqrt{p_i}}{k} < \frac{2\sqrt{p_n}}{k}$. Thus lemma \ref{lemma_a_1} is true.
\end{proof}

\begin{proof}[Proof of proposition \ref{prop_a_1}]
Let the condition of proposition \ref{prop_a_1} be satisfied. Let $n_0 > N_k$ be such that the interval $(n_0, 2n_0)$ contains less than $[(\frac{k}{4})\sqrt{2n_0}]$ primes. Let $p_{n-1}, p_n$ be such that $p_{n-1}<2n_0< p_n$. Then $(\frac{p_n}{2}, p_n)$ contains primes less than $[(\frac{k}{4})\sqrt{p_n}]$. Indeed, the interval $(n_0, p_n) = (n_0, 2n_0)\cup[2n_0]\cup(2n_0, p_n)$ contains primes less than $[(\frac{k}{4})\sqrt{2n_0}]$. Furthermore $(\frac{p_n}{2}, p_n)\subset(n_0, p_n)$ as $\frac{p_n}{2} > n_0$ so $(\frac{p_n}{2}, p_n)$ contains primes less than $[(\frac{k}{4})\sqrt{p_n}]$, contradicting lemma \ref{lemma_a_1}.
\end{proof}

\begin{corollary}[Bertrand's Postulate]\label{cor_a_1}
Let $\sqrt{p_m} - \sqrt{p_{m-1}} < 1$ be satisfied for every integer $m \geq 2$. Then $(n, 2n - 2)$ contains no less than two primes for every integer $n \geq 8$.
\end{corollary}

\begin{proof}
According to proposition \ref{prop_a_1}, where $k = 1, N_1 = 2$ an interval $(n, 2n)$ contains no less three prime numbers for every integer $n \geq 72$. Corollary \ref{cor_a_1} is true for all values of $n$ no less than $72$; by direct verification we find that it is true for smaller values. Thus corollary \ref{cor_a_1} is true for $n \geq 8$. 
\end{proof}

\begin{corollary}\label{cor_a_2}
Let $\sqrt{p_m} - \sqrt{p_{m-1}} < \frac{1}{2}$ be satisfied for every integer $m \geq 32$. Then an interval $(n, 2n)$ contains no less than $[\frac{1}{2}\sqrt{2n}]$ prime numbers for every integer $n \geq 2$.
\end{corollary}

\begin{proof}
This is a particular case of proposition \ref{prop_a_1} where $k = 2$, $N_2 = 131$. Since corollary \ref{cor_a_2} is true for all values of $n$ not less than $131$; by direct verification we find that it is true for smaller values. Thus corollary \ref{cor_a_2} is true for any $n \geq 2$.
\end{proof}

\noindent The theorem: ``An interval $(n, 2n)$ contains not less than $[\frac{1}{2}\sqrt{2n}]$ primes for every integer $n \geq 2$" has been proved by H. Karcher using Tschebyschef - Erdos approach \cite{Karcher:2011aa}.
\newline
\newline
The following statement is based on using Cramer's conjecture in the form $p_{n+1}- p_n= O(\ln^2 p_{n+1})$:

\begin{proposition}\label{prop_a_2}
Let Cramer's conjecture be true then there exist such constants $k, N$ such that for every integer $n > N$, an interval $(n, 2n)$ contains no less than $[(\frac{1}{k})\frac{n}{\log^22n}]$ primes.
\end{proposition}

\begin{lemma}\label{lemma_a_2}
Let there exist such such integers $k, N$ such that $p_m - p_{m-1}< k \log^2 p_m$ is true for every $p_m\geq N$. Then $(\frac{p_n}{2}, p_n)$ contains no less than $[\frac{1}{2k}\frac{p_n}{\log^2 p_n}]$ primes for every $p_n$ where $\frac{p_n}{2} > N$. 
\end{lemma}

\begin{proof}
The proof is the same as in lemma \ref{lemma_a_1}.
\end{proof}

\begin{proof}[Proof of proposition \ref{prop_a_2}]
Let the constants $k, N$ of proposition \ref{prop_a_2} such that $p_m - p_{m-1}< k \log^2p_m$ is true for every $p_m \geq N$. However, there is such an integer $n_0 > N$ that an interval $(n_0, 2n_0)$ contains primes less than $[\frac{1}{2k}\frac{2n_0}{\log^22n_0}]$. Let $p_{n-1}, p_n$ be such primes that $p_{n-1}<2n_0< p_n$. Then $(\frac{p_n}{2}, p_n)$ contains primes less than $[\frac{1}{2k}\frac{p_n}{\log^2p_n}]$. Since $(n_0, p_n) = (n_0, 2n_0)\cup[2n_0] \cup(2n_0, p_n)$ contains primes less than $[\frac{1}{2k}\frac{2n_0}{\log^22n_0}]$. Furthermore $(\frac{p_n}{2}, p_n)\subset(n_0, p_n)$ since $\frac{p_n}{2} > n_0 > N$ so $(\frac{p_n}{2}, p_n)$ contains primes less than $[\frac{1}{2k}\frac{p_n}{\log^2p_n}]$, contradicting lemma \ref{lemma_a_2}.
\end{proof}

\noindent We would like to note that Cramer's conjecture $p_{n+1}- p_n = O(\ln^2p_n)$ is consistent with the admissible estimate of the lower bound for the number of primes in the interval $(n, 2n), \sim \frac{n}{\log 2n}$, and the hypothesis $\sqrt{p_{n+1}} - \sqrt{p_n} = o(1)$ which has the experimental support \cite{Ribenboim:2004aa}. 
\newline
\newline
It is surprising that it is impossible to obtain the lower bound for the number of primes in $(n, 2n)$ in the classical form $\sim \frac{n}{\log 2n}$ by evaluations of the difference of primes. It is a real fact due to E. Westzynthius, $p_{m+1} - p_m = O (\log p_m)$ is not true. However, after works of P. Erd\"{o}s and R.Rankin it is expected that for any real $\epsilon > 0$ the relation $p_{m+1} - p_m = O(\log^{1+\epsilon}p_m)$ is true (if this is really so?) then the evaluation of the difference between consecutive primes permits to obtain the lower bound as $(\frac{1}{k(\epsilon)})(\frac{n}{log^{1+\epsilon} 2n})$ where $\frac{n}{log^{1+\epsilon} 2n} = O(\frac{n}{\log2n})$ under $\epsilon \rightarrow 0$ while  $k(\epsilon) = O(1)$ is not true. The conjecture $p_{n+1}- p_n = O(log^{1+\epsilon}p_n)$ is consistent both with the hypothesis $\sqrt{p_{n+1}} - \sqrt{p_n} = o(1)$ and with the admissible estimate of the lower bound for the number of primes in $(n, 2n)$.

\begin{proposition}\label{prop_a_3}
There exists a constant $C$ such that for every integer $n > C$ the interval $(n, 2n)$ contains not less than $[\frac{1}{2}(2n)^{0.475}]$ prime numbers. 
\end{proposition}

\begin{proof}
According to \cite{Baker:2001aa} ``Theorem 1. For all $x > x_0$, the interval $[x - x^{0. 525}, x]$ contains a prime number. With enough effort, the value of $x_0$ could be determined effectively".  We have that if $p_m > C = x_0$ then $p_m- p_{m-1}< p_m^{0.525}$. Further the proof goes like in proposition \ref{prop_a_2}.
\end{proof}

\noindent Nowadays the estimate $\sim n^{0.475}$ of the lower bound for the number of primes in $(n, 2n)$ is obtained by the evaluations of the difference between consecutive primes is the best available result under such an approach.

\section{Discussion and Conclusions}

We have shown that by the evaluations of the difference between consecutive primes one can obtain the estimates of the lower bound for the number of primes in an interval $(n, 2n)$. Nowadays the best available result under such an approach is $[\frac{1}{2}(2n)^{0.475}]$.
\newline
\newline
Our results permit us to conclude that the relations $p_{n+1} - p_n = O(\ln^2 p_n)$ (Cramer's conjecture) and $p_{n+1} - p_n = O(\log^{1+\epsilon}p_n)$ (ultra Cramer's conjecture) have real reasons to be valid as they are consistent both with the admissible estimate of the lower bound for a number of primes in $(n, 2n) \sim \frac{n}{\log2n}$ and with the conjecture $\sqrt{p_{n+1}} - \sqrt{p_n} = o(1)$ which has the experimental support \cite{Ribenboim:2004aa} and do not conflict with the results of the works of E. Westzynthius, P. Erd\"{o}s and R.Rankin.

\bibliography{references}

\begin{thebibliography}{1}

\bibitem{Baker:2001aa}
R.C. Baker, G.~Harman, and J.~Pintz.
\newblock The difference between consecutive primes ii.
\newblock {\em Proc. London Math. Soc.}, 83:532--562, 2001.

\bibitem{Cramer:1936aa}
Harald Cram{\'e}r.
\newblock On the order of magnitude of the difference between consecutive
  numbers.
\newblock {\em Acta Arith}, 2:23--46, 1936.

\bibitem{Granville:2010aa}
Andrew Granville.
\newblock Different approaches to the distribution of primes.
\newblock {\em Milan Journal of Mathematics}, 78(1):65--84, 2010.

\bibitem{Karcher:2011aa}
Hermann Karcher.
\newblock Bertrand's conjecture: At least one prime between $n$ and $2n$.
\newblock \url{http://www.math.uni-bonn.de/people/karcher/BertrandN_2N.pdf},
  2011.

\bibitem{Ribenboim:2004aa}
Paulo Ribenboim.
\newblock {\em The Little Book of Bigger Primes}.
\newblock Springer-Verlag, New York, 2004.

\bibitem{Trost:1968aa}
Ernst Trost.
\newblock {\em Primzahlen}.
\newblock Birkh{\"a}user-Verlag, 1968.

\end{thebibliography}
\bibliographystyle{plain}

\end{document}